\documentclass[11pt, twoside, leqno]{article}

\usepackage{amssymb}
\usepackage{amsmath}
\usepackage{color}
\usepackage{mathrsfs}
\usepackage{amsthm}
\usepackage{txfonts}

\usepackage{indentfirst}

\allowdisplaybreaks

\def\red{\color{red}}

\def\ls{\lesssim}
\def\gs{\gtrsim}
\def\fz{\infty}

\renewcommand{\r}{\right}
\newcommand{\lf}{\left}

\def\ls{\lesssim}
\def\gs{\gtrsim}

\def\rr{{\mathbb R}}

\def\rn{{{\rr}^n}}

\def\nn{{\mathbb N}}
\def\cc{{\mathbb C}}

\def\lz{\lambda}

\def\bz{\beta}

\def\fai{\varphi}
\def\gz{{\gamma}}

\def\ls{\lesssim}
\def\gs{\gtrsim}

\def\boz{\Omega}

\def\einf{\mathop\mathrm{\,ess\,inf\,}}

\def\hs{\hspace{0.3cm}}

\newtheorem{theorem}{Theorem}[section]
\newtheorem{lemma}[theorem]{Lemma}
\newtheorem{corollary}[theorem]{Corollary}

\theoremstyle{definition}
\newtheorem{remark}[theorem]{Remark}
\newtheorem{definition}[theorem]{Definition}
\newtheorem{assumption}[theorem]{Assumption}

\def\loc{{\mathop\mathrm{loc}}}

\numberwithin{equation}{section}

\begin{document}

\arraycolsep=1pt

\title{\Large\bf New Characterizations of Musielak--Orlicz--Sobolev Spaces via Sharp Ball Averaging Functions
\footnotetext{\hspace{-0.35cm} 2010 {\it Mathematics Subject
Classification}. {Primary 46E35; Secondary 42B35, 42B25.}
\endgraf{\it Key words and phrases}. Musielak--Orlicz--Sobolev space, variable exponent Sobolev space,
Orlicz--Sobolev space, sharp ball averaging function.
\endgraf This work is supported by the
National Natural Science Foundation  of China (Grant Nos. 11871254,
11571289, 11571039,  11761131002, 11726621, 11671185 and 11871100).}}
\author{Sibei Yang, Dachun Yang\,\footnote{Corresponding author/{\red October 4, 2018}/Final version.}\ \
and Wen Yuan}
\date{ }
\maketitle

\vspace{-0.6cm}

\begin{center}
\begin{minipage}{13.5cm}\small
{{\bf Abstract.}
In this article, the authors establish a new characterization of the Musielak--Orlicz--Sobolev space
on $\mathbb{R}^n$, which includes the classical Orlicz--Sobolev space, the weighted Sobolev space
and the variable exponent Sobolev space as special cases, in terms of sharp ball averaging functions.
Even in a special case, namely, the variable exponent Sobolev space, the obtained
result in this article improves the corresponding result obtained by P. H\"ast\"o and A. M. Ribeiro
[Commun. Contemp. Math. 19 (2017), 1650022, 13 pp] via weakening the assumption $f\in L^1(\mathbb R^n)$
into $f\in L_{\mathop\mathrm{loc}}^1(\mathbb R^n)$, which was conjectured to be true by H\"ast\"o and Ribeiro
in the aforementioned same article.}
\end{minipage}
\end{center}

\vspace{0.2cm}

\section{Introduction\label{s1}}

 Let $s\in(0,1)$ and $p\in[1,\fz)$. In what follows, we use the \emph{symbol $L^p(\rn)$}
to denote the space of all Lebesgue measurable functions $f$ on $\rn$ such that
$$\|f\|_{L^p(\rn)}:=\lf[\int_\rn|f(x)|^p\,dx\r]^{1/p}<\fz$$
and $L^p(\rn\times\rn)$ is similarly defined via replacing $\rn$ by $\rn\times\rn$.
The \emph{fractional order Sobolev space}
$W^{s,\,p}(\rn)$ is defined by setting
$$W^{s,\,p}(\rn):=\lf\{f\in L^p(\rn):\ \frac{f(x)-f(y)}{|x-y|^{\frac{n}{p}+s}}\in L^p(\rn\times\rn)\r\}
$$
equipped with the \emph{norm}
$$\|f\|_{W^{s,\,p}(\rn)}:=\lf\{\int_\rn\int_\rn\frac{|f(x)-f(y)|^p}{|x-y|^{n+sp}}\,dx\,dy\r\}^{1/p}.
$$
Bourgain, Brezis and Mironescu \cite{bbm01,bbm02} studied the limit behavior of the norm
$\|\cdot\|_{W^{s,\,p}(\rn)}$ as $s\uparrow1$, here and hereafter, the \emph{symbol} $s\uparrow1$
means that $s\in(0,1)$ increasingly converges to $1$ . More precisely, Bourgain, Brezis and Mironescu
\cite{bbm01} showed that, for any $f\in L^p(\rn)$, it holds true that
\begin{equation}\label{dq-1}
\lim_{s\uparrow1}(1-s)\int_\rn\int_\rn\frac{|f(x)-f(y)|^p}{|x-y|^{n+sp}}\,dx\,dy=C_{(n,\,p)}\int_\rn
|\nabla f(x)|^p\,dx,
\end{equation}
where $C_{(n,\,p)}$ is an explicit positive constant depending on $n$ and $p$.
Recently, these results have been generalized to the cases of the so-called magnetic space \cite{sv16}  and
the Orlicz--Sobolev space \cite{fs17}.

In what follows, we use the \emph{symbol $L^1_\loc(\rn)$} to
denote the space of all locally integrable functions on $\rn$.
Let $\mathcal{P}(\rn)$ be the set of all measurable functions
$p:\,\rn\to[1,\fz)$. For any $p\in\mathcal{P}(\rn)$, let
\begin{equation}\label{p-1}
p^+:=\mathop{\mathrm{ess\,sup}}\limits_{x\in\rn}p(x)\ \
\text{and} \ \ p^-:=\mathop{\mathrm{ess\,inf}}\limits_{x\in\rn}p(x).
\end{equation}
For any given $p\in\mathcal{P}(\rn)$,
the \emph{variable exponent modular} $\rho_{p(\cdot)}$
is defined by setting, for any $f\in L^1_\loc(\rn)$,
$$\rho_{p(\cdot)}(f):=\int_{\rn}|f(x)|^{p(x)}\,dx.
$$
The \emph{variable exponent Lebesgue space} $L^{p(\cdot)}(\rn)$ is defined by setting
\begin{eqnarray*}
L^{p(\cdot)}(\rn):&&=\big\{f\ \text{is measurable on}\ \rn:\ \\
&&\quad\quad\quad\quad\text{there exists}\ \lz\in(0,\fz)\
\text{such that}\ \rho_{p(\cdot)}(\lz f)<\fz\big\}
\end{eqnarray*}
equipped with the \emph{Luxemburg} (also called the {\it Luxembourg--Nakano}) \emph{norm}
$$\|f\|_{L^{p(\cdot)}(\rn)}:=\inf\lf\{\lz\in(0,\fz):\  \rho_{p(\cdot)}\lf(\frac{f}{\lz}\r)\le1\r\}.
$$
Moreover, the \emph{variable exponent Sobolev space} $W^{1,\,p(\cdot)}(\rn)$ is defined by setting
$$W^{1,\,p(\cdot)}(\rn):=\lf\{u\in L^{p(\cdot)}(\rn):\ |\nabla u|\in L^{p(\cdot)}(\rn)\r\}.
$$
For more studies on variable exponent Lebesgue spaces and Sobolev spaces,
we refer the reader to \cite{cf13,dhhr11}.

However, in the case of variable exponent Sobolev spaces, the similar characterization as in
\cite{bbm01,fs17,sv16} does not hold true, because that the translation operator may not
be bounded on the variable exponent Lebesgue space. Indeed, from \cite[Proposition 3.6.1]{dhhr11}, it follows that
the translation operator is bounded on the variable exponent Lebesgue space $L^{p(\cdot)}(\rn)$
if and only if $p$ is a constant. In particular, this means that, in the different quotient
appearing in \eqref{dq-1}, we can not replace the constant exponent $p$
simply by the variable exponent $p(x)$ or $p(y)$.
The same problem also appears in \cite{f02}. Instead of this, Diening and H\"ast\"o \cite{dh07}
creatively replaced the different quotient in the definition of the trace space
by the \emph{sharp averaging operator} $M^\sharp_{B}$ (see, for example, \cite[(1.3)]{hr17}
or \eqref{so} below). Motivated by the
work \cite{dh07}, under the assumptions that the variable exponent
$p(\cdot)$ satisfies the local log-H\"older continuity condition, the log-H\"older decay condition (at infinity)
and $p^-\in(1,\fz)$ with $p^-$ as in \eqref{p-1}, H\"ast\"o and Ribeiro
\cite[Theorem 4.1]{hr17} obtained a new characterization of
the variable exponent Sobolev space in terms of the sharp averaging operator, which is re-stated as
Theorem \ref{hr17thm} below.

As a natural generalization of the Lebesgue space, the
Orlicz space was introduced by Birnbaum and Orlicz \cite{bo31} and
Orlicz \cite{o32}. Since then, the theory of Orlicz spaces has been well
developed and these spaces have been widely used in probability, statistics, potential theory,
partial differential equations, as well as harmonic analysis and some
other fields of analysis (see, for example, \cite{am01,cm16,n50,n51,rr91,rr02}).
Later, Musielak \cite{m83} introduced the so-called Musielak--Orlicz space,
which contains the Orlicz space and the weighted Lebesgue space as special cases.
Nowadays, the theory of both Musielak--Orlicz spaces and function spaces of Musielak--Orlicz type
has been well developed and theses spaces have been widely used in many branches of mathematics.
It is worth pointing out that Musielak--Orlicz spaces or Musielak--Orlicz--Sobolev spaces
naturally appear in the study of the regularity for solutions of some nonlinear elliptic equations
or minimizers of functionals with non-standard growth (see, for example, \cite{am05,am01,cm16,cm15}).
We also refer the reader to \cite{os15,os16,ylk17,ya17}
for some recent progresses about the real-variable theory of both Musielak--Orlicz--Sobolev spaces and
function spaces of Musielak--Orlicz type.

In this article, motivated by \cite{dh07,fs17,hr17},
we obtain a new characterization of the Musielak--Orlicz--Sobolev space
on $\mathbb{R}^n$, including the classical Orlicz--Sobolev space, the weighted
Sobolev space and the variable exponent Sobolev space,
in terms of sharp ball averaging functions. Even in the special case, namely, the
variable exponent Sobolev space, the obtained
result in this article improves the corresponding result obtained in
\cite[Theorem 4.1]{hr17} via weakening the assumption $f\in L^1(\mathbb R^n)$
into $f\in L_\loc^1(\mathbb R^n)$, which positively confirms a conjecture proposed
by H\"ast\"o and Ribeiro in \cite[Remark 4.1]{hr17}.

To describe the main result of this article, we first recall some necessary notions and notation.
\begin{definition}\label{d1.1}
\begin{itemize}
\item[(i)] A function $G:\,[0,\fz)\to[0,\fz)$ is called an \emph{Orlicz function} if it satisfies
the following conditions:
\begin{itemize}
  \item[(i)$_1$] $G$ is continuous, convex, increasing and $G(0)=0$;
  \item[(i)$_2$] $\lim\limits_{t\to0^+}\frac{G(t)}{t}=0$ and $\lim\limits_{t\to\fz}\frac{G(t)}{t}=\fz$.
\end{itemize}

\item[(ii)] A function $\Phi:\,\rn\times[0,\fz)\to[0,\fz)$ is called a \emph{Musielak--Orlicz function}
if, for any given growth variable $t\in[0,\fz)$, $\Phi(\cdot,t)$ is measurable and, for almost every 
space variable $x\in\rn$, the function
$\Phi(x,\cdot)$ is an Orlicz function.

\item[(iii)] Let $\Phi$ be a Musielak--Orlicz function. Then the \emph{complementary function} of $\Phi$, denoted by
$\Phi^\ast$, is defined by setting, for any $x\in\rn$ and $s\in[0,\fz)$,
$$\Phi^\ast(x,s):=\sup_{t\in[0,\fz)}\{st-\Phi(x,t)\}.$$
\end{itemize}
\end{definition}

The Musielak--Orlicz space $L^\Phi(\rn)$ and the Musielak--Orlicz--Sobolev space
$W^{1,\,\Phi}(\rn)$ are defined as follows.

\begin{definition}\label{d1.2}
Let $\Phi$ be a Musielak--Orlicz function. For any given $f\in L^1_\loc(\rn)$,
the \emph{Musielak--Orlicz modular} of $f$ is defined by setting
$$\rho_\Phi(f):=\int_{\rn}\Phi(x,|f(x)|)\,dx.
$$
Then the \emph{Musielak--Orlicz space} $L^\Phi(\rn)$ is defined by setting
\begin{eqnarray*}
L^\Phi(\rn):&&=\big\{u\ \text{is measurable on}\ \rn:\ \\
&&\quad\quad\quad\quad\text{there exists}\ \lz\in(0,\fz)\
\text{such that}\ \rho_\Phi(\lz f)<\fz\big\}
\end{eqnarray*}
equipped with the \emph{Luxemburg} (also called the {\it Luxembourg--Nakano}) \emph{norm}
\begin{equation}\label{eq-Phi}
\|u\|_{L^\Phi(\rn)}:=\inf\lf\{\lz\in(0,\fz):\
\int_{\rn}\Phi\lf(x,\frac{|u(x)|}{\lz}\r)\,dx\le1\r\}.
\end{equation}
Moreover, the \emph{Musielak--Orlicz--Sobolev space} $W^{1,\,\Phi}(\rn)$ is defined by setting
$$W^{1,\,\Phi}(\rn):=\lf\{u\in L^\Phi(\rn):\ |\nabla u|\in L^\Phi(\rn)\r\}
$$
equipped with the \emph{norm}
$$\|u\|_{W^{1,\,\Phi}(\rn)}:=\|u\|_{L^\Phi(\rn)}+\||\nabla u|\|_{L^\Phi(\rn)},
$$
where $\||\nabla u|\|_{L^\Phi(\rn)}$ is defined via replacing $u$ by $|\nabla u|$ in \eqref{eq-Phi}.

Furthermore, the \emph{homogeneous Musielak--Orlicz Sobolev space} $\dot{W}^{1,\,\Phi}(\rn)$
is defined by setting
$$\dot{W}^{1,\,\Phi}(\rn):=\lf\{u\in L^1_\loc(\rn):\ |\nabla u|\in L^\Phi(\rn)\r\}
$$
equipped with the \emph{norm} $\|u\|_{\dot{W}^{1,\,\Phi}(\rn)}:=\||\nabla u|\|_{L^\Phi(\rn)}$.
\end{definition}

In what follows, for any $x\in\rn$ and $r\in(0,\fz)$, we always let
$$B(x,r):=\{y\in\rn:\ |y-x|<r\}$$
be a \emph{ball} of $\rn$ with the center $x$ and the radius $r$.

\begin{definition}\label{d1.3}
Let $\{\psi_{\epsilon}\}_{\epsilon>0}$ be a family of functions in $L^1([0,1])$ such that,
for any $\epsilon\in(0,\fz)$,
$$\psi_\epsilon\ge0,\quad\quad \int_0^1\psi_\epsilon(r)\,dr=1
$$
and, for any $\gamma\in(0,1)$,
$$\lim_{\epsilon\to0^+}\int_{\gamma}^1\psi_\epsilon(r)\,dr=0,
$$
here and hereafter, the \emph{symbol} $\epsilon\to0^+$ means that $\epsilon\in (0,\fz)$ and $\epsilon\to 0$.

Assume that $\Phi$ is a Musielak--Orlicz function.
For any $\epsilon\in(0,\fz)$ and $f\in L_{\loc}^1(\rn)$,
let
$$\rho^\epsilon_{\sharp}(f):=
\int_0^1\lf[\int_\rn\Phi\lf(x,\frac{1}{r}M^\sharp_{B(x,r)}(f)\r)\,dx\r]\psi_\epsilon(r)\,dr,
$$
where the \emph{sharp ball averaging function $M^\sharp_{B(x,r)}(f)$}
of $f$ is defined by setting, for any
$f\in L_{\loc}^1(\rn)$, $x\in\rn$ and $r\in(0,\fz)$,
\begin{equation}\label{so}
M^\sharp_{B(x,r)}(f):=\frac{1}{|B(x,r)|}\int_{B(x,r)}|f(y)-f_{B(x,r)}|\,dy
\end{equation}
and
$$f_{B(x,r)}:=\frac{1}{|B(x,r)|}\int_{B(x,r)}f(y)\,dy.
$$
Recall that $M^\sharp_{B(x,r)}$ is also called the \emph{sharp averaging operator} in
\cite[p.\,1650022-2]{hr17}.
Moreover, for any given $\epsilon\in(0,\fz)$ and $f\in L_{\loc}^1(\rn)$,
the \emph{norm} $\|f\|_{\Phi,\,\sharp}^\epsilon$ is defined by setting
$$\|f\|_{\Phi,\,\sharp}^\epsilon:=\inf\lf\{\lz\in(0,\fz):\
\rho^\epsilon_{\sharp}\lf(\frac{f}{\lz}\r)\le1\r\}.
$$
\end{definition}

We point out that several examples of such families of functions
$\{\psi_{\epsilon}\}_{\epsilon>0}$ as in
Definition \ref{d1.3} were given by Brezis \cite[Remark 8]{b02}.

Recall that a function $p:\ \rn\to\rr$ is said to satisfy the \emph{local
log-H\"older continuity condition} if there exists a positive constant $C$ such that, for any
$x,\,y\in\rn$ with $x\neq y$,
$$|p(x)-p(y)|\le\frac{C}{\log(e+\frac{1}{|x-y|})};
$$
a function $p:\ \rn\to\rr$ is said to satisfy the \emph{log-H\"older decay condition} (at infinity)
if there exist positive constants $C\in(0,\fz)$ and $p_\fz\in[1,\fz)$ such that, for any
$x\in\rn$,
$$|p(x)-p_\fz|\le\frac{C}{\log(e+|x|)}.
$$
If a function $p$ satisfies both the local log-H\"older continuity condition
and the log-H\"older decay condition,
then the function $p$ is said to satisfy the \emph{log-H\"older continuity condition}.
It is easy to see that, if $p$ satisfies the log-H\"older continuity condition, then $p$
is bounded and $p^+<\fz$ with $p^+$ as in \eqref{p-1}.

For the variable exponent Sobolev space,
the following conclusion was established in \cite[Theorem 4.1]{hr17}.

\setcounter{theorem}{0}
\renewcommand{\thetheorem}{\arabic{section}.\Alph{theorem}}

\begin{theorem}\label{hr17thm}
Let $p\in\mathcal{P}(\rn)$ satisfy the log-H\"older continuity condition and $p^-\in(1,\fz)$
with $p^-$ as in \eqref{p-1}, and $\{\psi_\epsilon\}_{\epsilon>0}$ be a family of
functions as in Definition \ref{d1.3}.
For any $x\in\rn$ and $t\in[0,\fz)$, let $\Phi(x,t):=t^{p(x)}$. Assume further that $f\in L^1(\rn)$.
Then $|\nabla f|\in L^{p(\cdot)}(\rn)$ if and only if
\begin{equation*}
\limsup_{\epsilon\to0^+}\rho^\epsilon_{\sharp}(f)<\fz,
\end{equation*}
here and hereafter, $\nabla f$ denotes the gradient of $f$.
In this case,
\begin{equation*}
\lim_{\epsilon\to0^+}\rho^\epsilon_{\sharp}(f)=\rho_{p(\cdot)}(c_0|\nabla f|)
\end{equation*}
and
\begin{equation*}
\lim_{\epsilon\to0^+}\|f\|_{\Phi,\,\sharp}^\epsilon=c_0\||\nabla f|\|_{L^{p(\cdot)}(\rn)},
\end{equation*}
where $c_0:=\frac{1}{|B(\vec{0}_{n},1)|}\int_{B(\vec{0}_{n},1)}|x\cdot e_1|\,dx$
with $\vec{0}_{n}$ being the \emph{origin} of $\rr^{n}$ and
$e_1:=(1,\,\overbrace{0,\,\ldots,\,0}^{n-1\ \mathrm{times}})$.
\end{theorem}

\setcounter{theorem}{0}\renewcommand{\thetheorem}{\arabic{section}.\arabic{theorem}}
\setcounter{theorem}{3}

\begin{remark}\label{hr17} We point out that, in \cite[Remark 4.1]{hr17},
H\"ast\"o and Ribeiro also conjectured that Theorem \ref{hr17thm} still holds true
if $f\in L^1(\rn)$ is replaced by $f\in L^1_\loc(\rn)$, which is confirmed
by Corollary \ref{c1.2}(i) below, as a simple corollary of Theorem \ref{t1.1} below.
\end{remark}

To state the main result of this article, we first give
some assumptions on the Musielak--Orlicz function $\Phi$.

\begin{assumption}\label{a1}
The Musielak--Orlicz function $\Phi$ is locally integrable on the space variable in $\rn$, namely,
for any given positive constant $c$ and any compact set $K\subset\rn$, we have
$$\int_{K}\Phi(x,c)\,dx<\fz.
$$
\end{assumption}

\begin{assumption}\label{a2}
The Musielak--Orlicz function $\Phi$ satisfies the \emph{$\Delta_2$-condition} on 
the growth variable, namely, there exists  a positive
constant $\kappa\in(1,\fz)$ such that, for almost every $x\in\rn$ and any $s\in[0,\fz)$,
$$\Phi(x,2s)\le\kappa\Phi(x,s).
$$
\end{assumption}

\begin{assumption}\label{a3}
The Musielak--Orlicz function $\Phi$ has the property:
$C^\fz_c(\rn)$ is dense in the homogeneous Musielak--Orlicz--Sobolev space $\dot{W}^{1,\Phi}(\rn)$
with respect to the norm $\|\cdot\|_{\dot{W}^{1,\Phi}(\rn)}$,
where the \emph{symbol} $C^\fz_c(\rn)$ denotes the set of all $C^\fz$ functions on $\rn$
with compact supports.
\end{assumption}

\begin{assumption}\label{a4}
The Hardy--Littlewood maximal operator $M$ is bounded on the
Musielak--Orlicz space $L^\Phi(\rn)$, namely,
there exists a positive constant $C$ such that, for any $f\in L^\Phi(\rn)$,
$$\|M(f)\|_{L^\Phi(\rn)}\le C\|f\|_{L^\Phi(\rn)}.
$$
Here and hereafter, $M(f)$ denotes the \emph{Hardy--Littlewood maximal function} of $f$,
which is defined by setting, for any given $f\in L^1_\loc(\rn)$ and $x\in\rn$,
$$M(f)(x):=\sup_{x\in B}
\frac{1}{|B|}\int_{B}|f(y)|\,dy,
$$
where the supremum  is taken over all balls $B\subset\rn$ containing $x$.
\end{assumption}

\begin{remark}\label{r1.1}
Here we give some examples of Musielak--Orlicz functions satisfying
Assumptions \ref{a3} and \ref{a4} as follows.
\begin{itemize}
  \item[\rm(i)] For any $x\in\rn$ and $t\in[0,\fz)$,
  let $\Phi(x,t):=t^{p(x)}$, where $p(\cdot)$ is as in Theorem \ref{hr17thm}.
  It is known that $\Phi$ satisfies Assumptions \ref{a3} and \ref{a4}
  (see, for example, \cite[Theorems 9.1.6 and 4.3.8]{dhhr11}).

  \item[\rm(ii)] For any $x\in\rn$ and $t\in[0,\fz)$, let $\Phi(x,t):=\fai(t)$,
  where $\fai$ is an Orlicz function satisfying the $\Delta_2$-condition.
  It is well known that Assumption \ref{a3} holds true for such $\Phi$
(see, for example, \cite[Theorem 8.31]{af03}). Moreover,
  by \cite[Theorem 2.1]{g88} (see also \cite[Theorem 1.2.1]{kk91}),
we find that $\Phi$ satisfies Assumption \ref{a4}.
  \item[\rm(iii)] Let $p\in[1,\fz)$.
Recall that an almost everywhere non-negative and locally integrable
function $\omega$ on $\rn$ is called an \emph{$A_p(\rn)$ weight} if
\begin{equation*}
[\omega]_{A_p(\rn)}:=\sup_{B\subset\rn}\lf\{\frac{1}{|B|}\int_{B}w(x)\,dx\r\}\lf\{
\frac{1}{|B|}\int_{B}[w(x)]^{-\frac{1}{p-1}}\,dx\r\}^{p-1}<\fz
\end{equation*}
when $p\in(1,\fz)$, and
\begin{equation*}
[\omega]_{A_1(\rn)}:=\sup_{B\subset\rn}\lf\{\frac{1}{|B|}\int_{B}w(x)\,dx\r\}
\lf\{\einf_{y\in B}w(y)\r\}^{-1}<\fz,
\end{equation*}
where the suprema are taken over all balls $B\subset\rn$.

For any $x\in\rn$ and $t\in[0,\fz)$, let $\Phi(x,t):=\omega(x)t^{q}$,
where $q\in(1,\fz)$ and $\omega\in A_q(\rn)$. From \cite[Theorem 2.1.4]{t00},
it follows that such $\Phi$ satisfies Assumption \ref{a3}.
Furthermore, it is well known that Assumption \ref{a4}
holds true for $\Phi$ (see, for example, \cite[Theorem 7.1.9]{g14}).

\item[\rm(iv)] More examples of Musielak--Orlicz functions satisfying Assumption \ref{a3}
were given in \cite{agsy17} (see also \cite[Example 2.1]{ya17}).
Moreover, some necessary and sufficient conditions for the Musielak--Orlicz function
$\Phi$ satisfying Assumption \ref{a4} were established in \cite{d05,h15}.
\end{itemize}
\end{remark}

Now we state the main result of this article as follows.

\begin{theorem}\label{t1.1}
Let $\Phi$ be a Musielak--Orlicz function satisfying Assumptions \ref{a1} through \ref{a4}.
Assume that the complementary function $\Phi^\ast$ to $\Phi$
satisfies Assumptions \ref{a1} and \ref{a2},
$\{\psi_\epsilon\}_{\epsilon>0}$ is a family of functions
as in Definition \ref{d1.3} and $f\in L^1_\loc(\rn)$. Then $|\nabla f|\in L^\Phi(\rn)$ if and only if
\begin{equation}\label{1.1}
\limsup_{\epsilon\to0^+}\rho^\epsilon_{\sharp}(f)<\fz.
\end{equation}
In this case,
\begin{equation}\label{1.2}
\lim_{\epsilon\to0^+}\rho^\epsilon_{\sharp}(f)=\rho_\Phi(c_0|\nabla f|)
\end{equation}
and
\begin{equation}\label{1.3}
\lim_{\epsilon\to0^+}\|f\|_{\Phi,\,\sharp}^\epsilon=c_0\||\nabla f|\|_{L^\Phi(\rn)},
\end{equation}
where $c_0$ is the same as in Theorem \ref{hr17thm}.
\end{theorem}

The detailed proof of Theorem \ref{t1.1} is presented in Section \ref{s2}.

To show Theorem \ref{t1.1}, we borrow some ideas from the proof of \cite[Theorem 4.1]{hr17}.
More precisely, we first prove that \eqref{1.2} holds true for functions in $C^\fz_c(\rn)$. Then,
for a function $f\in L^1_\loc(\rn)$ such that $|\nabla f|\in L^\Phi(\rn)$, by the technique of
approximation using functions from $C^\fz_c(\rn)$ and some finer properties of Musielak--Orlicz
functions, we show that \eqref{1.2} also holds true for such functions $f$.
Moreover, from \eqref{1.2} and the properties of Musielak--Orlicz
functions, we further deduce that \eqref{1.3} holds true. Comparing with \cite[Theorem 4.1]{hr17},
instead of $f\in L^1(\rn)$, we now only need to assume that $f\in L^1_\loc(\rn)$ in Theorem \ref{t1.1}.
To overcome the difficulty causing by this weaker assumption, in the proof of Theorem \ref{t1.1},
we flexibly use the boundedness of the Hardy--Littlewood maximal operator
on the Musielak--Orlicz space $L^\Phi(\rn)$,
which is assumed to hold true, and the subtle growth properties of Musielak--Orlicz functions
obtained in Lemma \ref{l2.2} below. Moreover, we point out that the assumed boundedness
of the Hardy--Littlewood maximal operator $M$ on the Musielak--Orlicz
space $L^\Phi(\rn)$ is known to hold true
for several well-known Musielak--Orlicz functions $\Phi$; see Corollary \ref{c1.2} below.

As a simple conclusion of Theorem \ref{t1.1}, we have the following conclusion,
the details being omitted here.

\begin{corollary}\label{c1.1}
Let $\Phi$ be a Musielak--Orlicz function satisfying Assumptions \ref{a1} through \ref{a4}.
Assume that the complementary function $\Phi^\ast$ to $\Phi$
satisfies Assumptions \ref{a1} and \ref{a2},
$\{\psi_\epsilon\}_{\epsilon>0}$ is a family of functions
as in Definition \ref{d1.3} and $f\in L^1_\loc(\rn)$.
Then $f\in W^{1,\,\Phi}(\rn)$ if and only if $f\in L^\Phi(\rn)$ and
$$
\limsup_{\epsilon\to0^+}\rho^\epsilon_{\sharp}(f)<\fz.
$$
Moreover, if $f\in W^{1,\,\Phi}(\rn)$, then
\begin{equation*}
\|f\|_{W^{1,\,\Phi}(\rn)}=\|f\|_{L^\Phi(\rn)}
+c_0^{-1}\lim_{\epsilon\to0^+}\|f\|_{\Phi,\,\sharp}^\epsilon,
\end{equation*}
where $c_0$ is the same as in Theorem \ref{hr17thm}.
 \end{corollary}

 Applying Theorem \ref{t1.1} and Corollary \ref{c1.1}, we can obtain the following conclusions
which are of independent interests.

\begin{corollary}\label{c1.2}
The conclusions of Theorem \ref{t1.1} and Corollary \ref{c1.1} hold true if $\Phi$
satisfies one of the following items:
\begin{itemize}
\item[\rm(i)] for any $x\in\rn$ and $t\in[0,\fz)$,
$\Phi(x,t):=t^{p(x)}$, where $p(\cdot)$ is as in Theorem
\ref{hr17thm}.
\item[\rm(ii)] for any $x\in\rn$ and $t\in[0,\fz)$, $\Phi(x,t):=\fai(t)$,  where $\fai$ is an
Orlicz function satisfying both the $\Delta_2$-condition and that there exist positive constants
$l\in(1,\fz)$ and $t_0\in[0,\fz)$ such that, for any $t\in[t_0,\fz)$,
\begin{equation}\label{1.4}
\fai(lt)\ge2l\fai(t).
\end{equation}
\item[\rm(iii)] for any $x\in\rn$ and $t\in[0,\fz)$, $\Phi(x,t):=\omega(x)t^{p}$,
where $p\in(1,\fz)$ and $\omega\in A_p(\rn)$.
\item[\rm(iv)] for any $x\in\rn$ and $t\in[0,\fz)$, $\Phi(x,t):=t^{p}+\omega(x)t^{q}$,
where $1<p<q<\fz$ and $\omega\in A_q(\rn)$.
\end{itemize}
 \end{corollary}

The proof of Corollary \ref{c1.2} is also presented in Section \ref{s2}.

\begin{remark}\label{r1.2}
\begin{itemize}
\item[(i)] We point out that Corollary \ref{c1.2}(i) improves Theorem \ref{hr17thm}
via weakening the assumption $f\in L^1(\mathbb R^n)$
into $f\in L_\loc^1(\mathbb R^n)$, which positively confirms a conjecture proposed
by H\"ast\"o and Ribeiro in \cite[Remark 4.1]{hr17} (see also Remark \ref{hr17}).

\item[(ii)] The condition \eqref{1.4} in Corollary \ref{c1.2}(ii)
is to guarantee that the complementary function $\fai^\ast$ to $\fai$
satisfies the $\Delta_2$-condition.
A typical example of such a $\fai$ is that, for any $t\in[0,\fz)$,
$\fai(t):=t^p(|\log t|+1)$, where $p\in(1,\fz)$ is a positive
constant (see, for example, \cite[p.\,27, (4.13)]{kr61}).

\item[(iii)] We point out that Musielak--Orlicz functions $\Phi$ as in Corollary \ref{c1.2}(iv)
and the corresponding Musielak--Orlicz--Sobolev spaces naturally appear in
the study on the double phase variational problems (see, for example, \cite{agsy17,cm15}).
\end{itemize}
\end{remark}

\begin{remark}\label{r1.3}
Some time later after we have completed the first version of this article,
we learned from Professor Peter H\"ast\"o that a result similar to
Theorem \ref{t1.1} had been independently obtained by Ferreira et al. \cite{fhr16}
for the homogeneous Musielak--Orlicz--Sobolev spaces
$\dot{W}^{1,\,\Phi}(\rn)$ and $\dot{W}^{1,\,\Phi}(\boz)$
under the quite different assumptions for $\Phi$, where $\boz$ is an open set in $\rn$.

More precisely, let $\fai:\ [0,\fz]\to[0,\fz]$ be an increasing function.
Denote by $\fai^{-1}:\ [0,\fz]\to[0,\fz]$ the
\emph{left-continuous generalized inverse} of $\fai$,
namely, for any $s\in[0,\fz]$,
$$\fai^{-1}(s):=\inf\lf\{t\in[0,\fz]:\ \fai(t)\ge s\r\}.
$$
Moreover, a function $g:\ \rr\to[0,\fz]$ is said to be \emph{almost increasing} if
there exists a positive constant $c$ such that, for any $t_1,\,t_2\in\rr$ with $t_1\le t_2$,
$g(t_1)\le cg(t_2)$.

Let the Musielak--Orlicz function $\Phi$ satisfy the following assumptions:
\begin{itemize}
  \item[\rm(a)] $\Phi$ satisfies the $\Delta_2$-condition.
  \item[\rm(b)] There exists a positive constant $\ell\in(1,\fz)$ such that, for almost
every $x\in\rn$, the function $s\mapsto s^{-\ell}\Phi(x,s)$ in $(0,\fz)$ is almost increasing
with the constant $\ell$ independent of $x$.
  \item[\rm(c)] There exist positive constants $\beta\in(0,1)$ and $\gamma\in(0,\fz)$ such that
\begin{itemize}
  \item[\rm(c)$_1$] for any $x\in\rn$, $\Phi(x,\beta\gamma)\le1\le\Phi(x,\gamma)$;
  \item[\rm(c)$_2$] for any ball $B\subset\rn$, any $x,\,y\in B$
and $t\in[\gamma,\Phi^{-1}(y,|B|^{-1})]$, $\Phi(x,\beta t)\le\Phi(y,t)$;
  \item[\rm(c)$_3$] there exists a function $h\in L^1(\rn)\cap L^\fz(\rn)$ satisfying that,
for almost every $x,\,y\in\rn$ and any $t\in[0,\gamma]$,
$$\Phi(x,\beta t)\le\Phi(y,t)+h(x)+h(y).
$$
\end{itemize}
\end{itemize}
For such a Musielak--Orlicz function $\Phi$, the conclusion of Theorem \ref{t1.1} was
established in \cite[Theorem 1.1]{fhr16}. Moreover, replaced $\rn$ by an open set $\boz\subset\rn$,
the conclusion of Theorem \ref{t1.1} was also obtained in
\cite[Theorem 1.1]{fhr16} under the assumptions $\rm(a)$, $\rm(b)$
and $\rm(c)$ with replacing $\rn$ by $\boz$.

We point out that, in the case of the Euclidean space
$\rn$, Theorem \ref{t1.1} completely covers \cite[Theorem 1.1]{fhr16};
in other words, the assumptions in Theorem \ref{t1.1} are \emph{weaker} than
the assumptions of \cite[Theorem 1.1]{fhr16}.
Indeed, by \cite[Theorem 4.6]{h15}, we know that the above assumptions $\rm(b)$
and $\rm(c)$ imply Assumption \ref{a4}. Moreover, from \cite[Theorem 6.5]{hhk16}
and \cite[Theorem 4.6]{h15}, it follows that the above assumptions $\rm(a)$, $\rm(b)$
and $\rm(c)$ imply Assumption \ref{a3}. Furthermore, it is easy to see that
the above assumptions $\rm(a)$ and $\rm(c)_1$ imply Assumption \ref{a1}.
By \cite[Lemma 2.4(2)]{hht17} and \cite[Lemma 2.4]{hhk16}, we conclude that
the above assumptions $\rm(a)$ and $\rm(b)$ imply that
the complementary function $\Phi^\ast$ to $\Phi$
satisfies Assumption \ref{a2}, which, combined with \cite[Proposition 4.6]{hhk16}
and the above assumptions $\rm(a)$ and $\rm(c)_1$, further implies that
the complementary function $\Phi^\ast$ satisfies Assumption \ref{a1}.
Thus, the above assumptions $\rm(a)$, $\rm(b)$
and $\rm(c)$ for $\Phi$ imply that Assumptions \ref{a1} through \ref{a4}
hold true for $\Phi$ and that the complementary function $\Phi^\ast$ to $\Phi$
satisfies Assumptions \ref{a1} and \ref{a2}, which further implies that
Theorem \ref{t1.1} completely covers \cite[Theorem 1.1]{fhr16} in the case of $\rn$.

We point out that (i) and (ii) of Corollary \ref{c1.2} are also simple corollaries
of \cite[Theorem 1.1]{fhr16}. However, the function $\Phi$ as in (iii) and (iv)
of Corollary \ref{c1.2} may not satisfy the assumption (c)$_1$.
Thus, (iii) and (iv) of Corollary \ref{c1.2} cannot be deduced from
\cite[Theorem 1.1]{fhr16}, but they can be deduced from Theorem \ref{t1.1}. In this sense,
Theorem \ref{t1.1} strictly has \emph{more generality} than \cite[Theorem 1.1]{fhr16} 
in the case of the Euclidean space $\rn$.
\end{remark}

Finally, we make some conventions on notation. Throughout the
article, we always denote by $C$ a \emph{positive constant} which is
independent of the main parameters, but it may vary from line to
line. We also use $C_{(\gz,\bz,\ldots)}$ to denote a  \emph{positive
constant} depending on the indicated parameters $\gz,$ $\bz$,
$\ldots$. The \emph{symbol} $f\ls g$ means that $f\le Cg$. If $f\ls
g$ and $g\ls f$, then we write $f\sim g$. We also use the following
convention: If $f\le Cg$ and $g=h$ or $g\le h$, we then write $f\ls g\sim h$
or $f\ls g\ls h$, \emph{rather than} $f\ls g=h$
or $f\ls g\le h$. For any measurable subset $E$ of $\rn$, we denote by $\mathbf{1}_{E}$ its
\emph{characteristic function}. We also let $\nn:=\{1,\, 2,\, \ldots\}$
and use $\vec 0_n$ to denote the origin of $\rn$.

\section{Proofs of Theorem \ref{t1.1} and Corollary \ref{c1.2}\label{s2}}

In this section, we give the proofs of
Theorem \ref{t1.1} and Corollary \ref{c1.2}.
We begin with some auxiliary conclusions. In what follows, we use the \emph{symbol} $C^2_c(\rn)$
to denote the set of all functions having continuous
derivatives till order 2 with compact supports.

\begin{lemma}\label{l2.1}
Assume that the Musielak--Orlicz function $\Phi$ satisfies Assumptions \ref{a1} and \ref{a2}.
If $f\in C^2_c(\rn)$, then
$$\lim_{\epsilon\to0^+}\rho^\epsilon_{\sharp}(f)=\rho_\Phi(c_0|\nabla f|),
$$
where $c_0$ is the same as in Theorem \ref{hr17thm}.
\end{lemma}

To prove Lemma \ref{l2.1}, we need the following properties of Musielak--Orlicz functions.

\begin{lemma}\label{l2.2}
Assume that the Musielak--Orlicz function $\Phi$ satisfies Assumption \ref{a2}.
\begin{itemize}
\item[\rm (i)] For almost every $x\in\rn$, any $a\in(0,\fz)$
and any $b\in(0,1)$, $\Phi(x,ab)\le b\Phi(x,a)$.
\item[\rm (ii)] There exists a positive constant $\gamma\in(\kappa,\fz)$
such that, for almost every $x\in\rn$, any $a\in(0,\fz)$ and any $b\in[1,\fz)$,
$\Phi(x,ab)\le b^\gamma\Phi(x,a)$,
where $\kappa$ is as in Assumption \ref{a2}.
\item[\rm (iii)] For any $\delta\in(0,\fz)$, there exists a positive constant $C_{(\delta)}$,
depending on $\delta$, such that, for almost every $x\in\rn$ and any $a,\,b\in(0,\fz)$,
$$\Phi(x,a+b)\le (1+\delta)^\gamma\Phi(x,a)+C_{(\delta)}\Phi(x,b),
$$
where $\gamma$ is as in Lemma \ref{l2.2}(ii).
\end{itemize}
\end{lemma}

\begin{proof}
We first show (i). By the fact that, for almost every $x\in\rn$,
$\Phi(x,\cdot)$ is convex and $\Phi(x,0)=0$, we find that,
for almost every $x\in\rn$, any $a\in(0,\fz)$ and any $b\in(0,1)$,
$\Phi(x,ab)\le b\Phi(x,a)$, which implies that (i) holds true.

Now we give the proof of (ii). From the proof of \cite[Lemma 2.4(2)]{hht17},
it follows that there exists a positive constant $\gamma\in(\kappa,\fz)$
such that, for almost every $x\in\rn$, the function $t\mapsto t^{-\gamma}\Phi(x,t)$
is decreasing in $(0,\fz)$,
which further implies that, for almost every $x\in\rn$,
any $a\in(0,\fz)$ and any $b\in[1,\fz)$,
$(ab)^{-\gamma}\Phi(x,ab)\le a^{-\gamma}\Phi(x,a)$.
By this, we find that, for  almost every
$x\in\rn$, any $a\in(0,\fz)$ and any $b\in[1,\fz)$,
$$\Phi(x,ab)\le b^\gamma\Phi(x,a).$$

Finally, we prove (iii). Let $\delta\in(0,\fz)$. Fix $x\in\rn$ and $a,b\in[0,\fz)$.
If $b>\delta a$, from the fact that $\Phi(x,\cdot)$ is increasing and Assumption \ref{a2},
we deduce that
\begin{equation}\label{2.2}
\Phi(x,a+b)\le\Phi\lf(x,b\lf(1+\frac{1}{\delta}\r)\r)\le\Phi\lf(x,2^mb\r)\le\kappa^m\Phi(x,b),
\end{equation}
where $\kappa$ is as in Assumption \ref{a2} and $m:=m(\delta)\in\nn$ satisfies that $1+\frac{1}{\delta}\le2^m$.
Moreover, if $b\le\delta a$, by (ii), we conclude that
$$\Phi(x,a+b)\le\Phi(x, a(1+\delta))\le(1+\delta)^{\gamma}\Phi(x,a),
$$
which, combined with \eqref{2.2}, then completes the proof of (iii) and hence of Lemma \ref{l2.2}.
\end{proof}

Now we show Lemma \ref{l2.1} by using Lemma \ref{l2.2}.

\begin{proof}[Proof of Lemma \ref{l2.1}]
Let $f\in C^2_c(\rn)$. Then, by the Taylor expansion, we know that, for any $x,\,y\in\rn$,
$$f(y)=f(x)+\nabla f(x)\cdot(y-x)+R(x,y),
$$
where $R(x,y)=o(|x-y|)$ as $y\to x$, which implies that, for any $x,\,y\in\rn$ and $r\in(0,\fz)$,
\begin{align}\label{2.3}
&f(y)-f_{B(x,r)}\\
&\hs=\frac{1}{|B(x,r)|}\int_{B(x,r)}[f(y)-f(z)]\,dz\nonumber\\
&\hs=\frac{1}{|B(x,r)|}\int_{B(x,r)}[\nabla f(x)\cdot(y-x)-\nabla f(x)\cdot(z-x)+R(x,y)-R(x,z)]\,dz.\nonumber
\end{align}
From symmetry, it follows that
$$\int_{B(x,r)}\nabla f(x)\cdot(z-x)\,dz=0,
$$
which, together with \eqref{2.3}, further implies that, for any $x,\,y\in\rn$ and $r\in(0,\fz)$,
\begin{align}\label{2.4}
f(y)-f_{B(x,r)}=\nabla f(x)\cdot(y-x)+R(x,y)-\frac{1}{|B(x,r)|}\int_{B(x,r)}R(x,z)\,dz.
\end{align}
By \eqref{2.4} and the triangle inequality on $\cc$, we conclude that,
for any $x,\,y\in\rn$ and $r\in(0,\fz)$,
\begin{align}\label{2.5}
M^\sharp_{B(x,r)}(f)&\le\frac{1}{|B(x,r)|}\int_{B(x,r)}|\nabla f(x)\cdot(y-x)|\,dy
+\frac{2}{|B(x,r)|}\int_{B(x,r)}|R(x,y)|\,dy\\
&=\frac{|\nabla f(x)|}{|B(x,r)|}\int_{B(x,r)}
\lf|\frac{\nabla f(x)}{|\nabla f(x)|}\cdot(y-x)\r|\,dy
+\frac{2}{|B(x,r)|}\int_{B(x,r)}|R(x,y)|\,dy.\nonumber
\end{align}
Noticing that $\int_{B(x,r)}|\nu\cdot(y-x)|\,dy$ is
independent of $\nu\in S^{n-1}$, from \eqref{2.5},
we further deduce that, for any $x,\,y\in\rn$ and $r\in(0,\fz)$,
\begin{align}\label{2.6}
M^\sharp_{B(x,r)}(f)\le c_0 r|\nabla f(x)|+\frac{2}{|B(x,r)|}\int_{B(x,r)}|R(x,y)|\,dy,
\end{align}
where $c_0$ is the same as in Theorem \ref{hr17thm}. Similarly to \eqref{2.6}, we also have
\begin{align}\label{2.7}
M^\sharp_{B(x,r)}(f)\ge c_0 r|\nabla f(x)|-\frac{2}{|B(x,r)|}\int_{B(x,r)}|R(x,y)|\,dy.
\end{align}
Let $\boz:=\{x\in\rn:\ f|_{B(x,1)} \not\equiv 0\}$. Then
$$\rho^{\epsilon}_{\sharp}(f)=\int_0^1\lf[\int_{\boz}\Phi\lf(x,\frac{1}{r}
M^{\sharp}_{B(x,r)}(f)\r)\,dx\r]\psi_{\epsilon}(r)\,dr.
$$
By $f\in C^2_c(\rn)$, we know that $\boz$ is a bounded set in $\rn$.
From \eqref{2.6}, \eqref{2.7} and the fact that $M^{\sharp}_{B(x,r)}(f)\ge0$, it follows that
\begin{align}\label{2.8}
&\int_0^1\lf\{\int_\boz\Phi(x,\max\{0,c_0|\nabla f(x)|-h(x,r)\})\,dx\r\}\psi_{\epsilon}(r)\,dr\\
&\hs\le\rho^{\epsilon}_{\sharp}(f)
\le\int_0^1\lf\{\int_\boz\Phi(x,c_0|\nabla f(x)|+h(x,r))\,dx\r\}\psi_{\epsilon}(r)\,dr,\nonumber
\end{align}
where $h(x,r):=\frac{2}{r}\frac{1}{|B(x,r)|}\int_{B(x,r)}|R(x,y)|\,dy$.

By Lemma \ref{l2.2}(iii), we conclude that, for any $\delta\in(0,\fz)$,
there exists a positive constant
$C_{(\delta)}$, depending on $\delta$, such that, for almost every $x\in\rn$,
$$\Phi(x,c_0|\nabla f(x)|+h(x,r))\le(1+\delta)^{\gamma}\Phi(x,c_0|\nabla f(x)|)+C_{(\delta)}\Phi(x,h(x,r)),
$$
which, together with $\int_0^1\psi_\epsilon(r)\,dr=1$, further implies that
\begin{align}\label{2.9}
\rho^{\epsilon}_{\sharp}(f)
\le(1+\delta)^{\gamma}\int_\boz\Phi(x,c_0|\nabla f(x)|)\,dx+C_{(\delta)}
\int_0^1\lf[\int_\boz\Phi(x,h(x,r))\,dx\r]\psi_{\epsilon}(r)\,dr.
\end{align}
Letting $\epsilon\to0^+$ in \eqref{2.9}, we find that
\begin{align}\label{2.10}
\limsup_{\epsilon\to0^+}\rho^{\epsilon}_{\sharp}(f)
\le(1+\delta)^{\gamma}\rho_\Phi(c_0|\nabla f|)+\limsup_{\epsilon\to0^+}C_{(\delta)}
\int_0^1\lf[\int_\boz\Phi(x,h(x,r))\,dx\r]\psi_{\epsilon}(r)\,dr.
\end{align}
Assume that $\epsilon\in(0,1/2)$.
Let $\sigma\in(0,\fz)$ be such that $\frac{|R(x,y)|}{|x-y|}<\epsilon$ when $|x-y|<\sigma$.
By $f\in C^2_c(\rn)$, we know that $\frac{|R(x,y)|}{|x-y|}$ is bounded,
which further implies that
there exists a positive constant $C$ such that
\begin{align*}
h(x,r)&=\frac{2}{r}\frac{1}{|B(x,r)|}\int_{B(x,r)}|R(x,y)|\,dy\\
&\le\frac{2}{|B(x,r)|}\int_{B(x,r)}\frac{|R(x,y)|}{|x-y|}\,dy
\le2\epsilon\mathbf{1}_{(0,\sigma)}(r)+C\mathbf{1}_{(\sigma,1)}(r).
\end{align*}
From this, the definition of $\psi_\epsilon$ and Lemma \ref{l2.2}(i), we deduce that
\begin{align*}
&\int_0^1\lf[\int_\boz\Phi(x,h(x,r))\,dx\r]\psi_{\epsilon}(r)\,dr\\
&\hs\le\int_\boz\lf[\int_0^\sigma\Phi(x,2\epsilon)\psi_{\epsilon}(r)\,dr
+\int_\sigma^1\Phi(x,C)\psi_{\epsilon}(r)\,dr\r]\,dx\\
&\hs\le2\epsilon\int_\boz\Phi(x,1)\,dx+\int_\boz\Phi(x,C)\,dx\int_\sigma^1\psi_{\epsilon}(r)\,dr,
\end{align*}
which, combined with Assumption \ref{a1} and the definition of $\psi_\epsilon$, further implies that
\begin{align}\label{2.11}
\lim_{\epsilon\to0^+}\int_0^1\lf[\int_\boz\Phi(x,h(x,r))\,dx\r]\psi_{\epsilon}(r)\,dr=0.
\end{align}
By this and \eqref{2.10}, we conclude that
\begin{align}\label{2.12}
\limsup_{\epsilon\to0^+}\rho^{\epsilon}_{\sharp}(f)
\le(1+\delta)^{\gamma}\rho_\Phi(c_0|\nabla f|).
\end{align}
Letting $\delta\to0^+$ in \eqref{2.12}, then we know that
\begin{align}\label{2.13}
\limsup_{\epsilon\to0^+}\rho^{\epsilon}_{\sharp}(f)
\le\rho_\Phi(c_0|\nabla f|).
\end{align}
Moreover, from Lemma \ref{l2.2}(iii), we deduce that, for any $\delta\in(0,\fz)$,
there exists a positive constant $C_{(\delta)}$, depending on $\delta$, such that,
for almost every $x\in\rn$ and any $r\in(0,\fz)$,
\begin{align*}
&\Phi(x,\max\{0,c_0|\nabla f(x)|-h(x,r)\})\\
&\hs\ge(1+\delta)^{-\gamma}\Phi(x,c_0|\nabla f(x)|)
-C_{(\delta)}(1+\delta)^{-\gamma}\Phi(x,h(x,r)),
\end{align*}
which, together with \eqref{2.8}, implies that
\begin{align*}
&\liminf_{\epsilon\to0^+}\rho^{\epsilon}_{\sharp}(f)\\
&\hs\ge(1+\delta)^{-\gamma}\rho_\Phi(c_0|\nabla f|)
-\limsup_{\epsilon\to0^+}C_{(\delta)}(1+\delta)^{-\gamma}
\int_0^1\lf[\int_\boz\Phi(x,h(x,r))\,dx\r]\psi_{\epsilon}(r)\,dr.
\end{align*}
By this and \eqref{2.11}, similarly to \eqref{2.13}, we conclude that
\begin{align*}
\liminf_{\epsilon\to0^+}\rho^{\epsilon}_{\sharp}(f)
\ge\rho_\Phi(c_0|\nabla f|),
\end{align*}
which, combined with \eqref{2.13}, further implies that
$$\lim_{\epsilon\to0^+}\rho^{\epsilon}_{\sharp}(f)
=\rho_\Phi(c_0|\nabla f|).
$$
This finishes the proof of Lemma \ref{l2.1}.
\end{proof}

\begin{lemma}\label{l2.3}
Assume that the Musielak--Orlicz function $\Phi$ satisfies Assumptions \ref{a2} and \ref{a4}.
Then there exists a positive constant $C$ such that, for any $\epsilon\in(0,\fz)$ and
$f\in \dot{W}^{1,\,\Phi}(\rn)$,
$$\rho^\epsilon_{\sharp}(f)\le C\max\lf\{\||\nabla f|\|_{L^\Phi(\rn)},\,
\||\nabla f|\|_{L^\Phi(\rn)}^\gamma\r\},
$$
where $\gamma$ is as in Lemma \ref{l2.2}(ii).
\end{lemma}

\begin{proof}
First, let $f\in\dot{W}^{1,\,\Phi}(\rn)$ be such that $\||\nabla f|\|_{L^\Phi(\rn)}\le1$.
By the Poincar\'e inequality, we know that, for any $x\in\rn$ and $r\in(0,\fz)$,
$$M^{\sharp}_{B(x,r)}(f)=\frac{1}{|B(x,r)|}\int_{B(x,r)}|f(y)-f_{B(x,r)}|\,dy\ls
\frac{r}{|B(x,r)|}\int_{B(x,r)}|\nabla f(y)|\,dy,
$$
which, together with the fact that,
for any $\epsilon\in(0,\fz)$, $\int_0^1\psi_\epsilon(r)\,dr=1$, further implies that
\begin{align}\label{2.14}
\rho^{\epsilon}_{\sharp}(f)
&\ls\int_0^1\lf[\int_{\rn}\Phi\lf(x,\frac{1}{|B(x,r)|}\int_{B(x,r)}
|\nabla f(y)|\,dy\r)\,dx\r]\psi_{\epsilon}(r)\,dr\\
&\ls\int_0^1\lf[\int_{\rn}\Phi\lf(x,M(|\nabla f|)(x)\r)\,dx\r]\psi_{\epsilon}(r)\,dr\nonumber\\
&\sim\int_{\rn}\Phi\lf(x,M(|\nabla f|)(x)\r)\,dx.\nonumber
\end{align}
Moreover, from (i) and (ii) of Lemma \ref{l2.2}, Assumption \ref{a4}
and $\||\nabla f|\|_{L^\Phi(\rn)}\le1$,
we deduce that
\begin{align*}
\int_{\rn}\Phi\lf(x,M(|\nabla f|)(x)\r)\,dx&\ls\max\lf\{\|M(|\nabla f|)\|_{L^\Phi(\rn)},\,
\|M(|\nabla f|)\|_{L^\Phi(\rn)}^\gamma\r\}\\
&\ls\max\lf\{\||\nabla f|\|_{L^\Phi(\rn)},\,
\||\nabla f|\|_{L^\Phi(\rn)}^\gamma\r\}\ls1,
\end{align*}
which, combined with \eqref{2.14}, implies that
\begin{align}\label{2.15}
\rho^{\epsilon}_{\sharp}(f)\ls\int_{\rn}\Phi\lf(x,M(|\nabla f|)(x)\r)\,dx\ls1.
\end{align}

For any $f\in\dot{W}^{1,\,\Phi}(\rn)$, replacing $f$ by $f/\||\nabla f|\|_{L^\Phi(\rn)}$ and repeating
the proof of \eqref{2.15}, we know that
\begin{align}\label{2.16}
\rho^{\epsilon}_{\sharp}\lf(\frac{f}{\||\nabla f|\|_{L^\Phi(\rn)}}\r)
\ls\int_{\rn}\Phi\lf(x,\frac{M(|\nabla f|)(x)}{\||\nabla f|\|_{L^\Phi(\rn)}}\r)\,dx\ls1.
\end{align}
Moreover, from the definition of $\rho^{\epsilon}_{\sharp}$ and (i) and (ii) of Lemma \ref{l2.2},
it follows that, for any $\lz\in(0,\fz)$,
$$\rho^{\epsilon}_{\sharp}(f)\le\rho^{\epsilon}_{\sharp}
\lf(\frac{f}{\lz}\r)\max\{\lz,\,\lz^\gamma\},
$$
where $\gamma$ is as in Lemma \ref{l2.2}(ii),
which, combined with \eqref{2.16}, further implies that
$$\rho^\epsilon_{\sharp}(f)\ls\max\lf\{\||\nabla f|\|_{L^\Phi(\rn)},\,
\||\nabla f|\|_{L^\Phi(\rn)}^\gamma\r\}.
$$
This finishes the proof of Lemma \ref{l2.3}.
\end{proof}

To prove Theorem \ref{t1.1}, we need the following reflexivity
of the Musielak--Orlicz space $L^\Phi(\rn)$,
which was obtained in \cite[Theorem 1.4]{ya17}.

\begin{lemma}\label{l2.4}
Let $\Phi$ be a Musielak--Orlicz function and $\Phi^\ast$ be the complementary function
to $\Phi$. If both $\Phi$ and $\Phi^\ast$ satisfy Assumptions \ref{a1} and \ref{a2},
then the Musielak--Orlicz space $L^\Phi(\rn)$ is reflexive, namely,
$(L^\Phi(\rn))^{\ast\ast}=L^\Phi(\rn)$, where $(L^\Phi(\rn))^{\ast}$ denotes the dual space
of $L^\Phi(\rn)$ [namely, the space of all continuous linear functions on $L^\Phi(\rn)$] and
$$\lf(L^\Phi(\rn)\r)^{\ast\ast}=\lf(\lf(L^\Phi(\rn)\r)^{\ast}\r)^{\ast}.$$
\end{lemma}

Let $G\in C^\fz_c(B(\vec{0}_n,1))$ be a standard mollifier. Namely, for any $x\in\rn$,
\begin{equation}\label{2.17}
G(x):=\begin{cases}
C_1e^{-\frac{1}{1-|x|^2}}\ \ &\text{if}\ |x|<1,\\
0\  &\text{if}\ |x|\ge1,
\end{cases}
\end{equation}
where $C_1$ is a positive constant such that $\int_{\rn}G(x)\,dx=1$.
For any $\varepsilon\in(0,\fz)$ and $x\in\rn$,
let $G_\varepsilon(x):=\varepsilon^{-n}G(x/\varepsilon)$.

Now we prove Theorem \ref{t1.1} by using Lemmas \ref{l2.1} through \ref{l2.4}.

\begin{proof}[Proof of Theorem \ref{t1.1}]
We divide the proof into the following three steps according to
the sufficient and the necessary conditions for \eqref{1.1} and
the equivalence of \eqref{1.2} and \eqref{1.3}.

\emph{Step 1)} In this step, we show that, if $|\nabla f|\in L^\Phi(\rn)$,
then \eqref{1.2} holds true,
which further implies that \eqref{1.1} also holds true.

By the triangle inequality on $\cc$, we find that, for any $x\in\rn$,
$r\in(0,\fz)$ and $f,\,g\in L^1_\loc(\rn)$,
\begin{align*}
&\frac{1}{|B(x,r)|}\int_{B(x,r)}|f(y)-f_{B(x,r)}|\,dy\\
&\hs\le\frac{1}{|B(x,r)|}\int_{B(x,r)}|f(y)-g(y)-(f-g)_{B(x,r)}|\,dy
+\frac{1}{|B(x,r)|}\int_{B(x,r)}|g(y)-g_{B(x,r)}|\,dy,
\end{align*}
which implies that, for any $x\in\rn$, $r\in(0,\fz)$ and $f,\,g\in L^1_\loc(\rn)$,
$$M^{\sharp}_{B(x,r)}(f)\le M^{\sharp}_{B(x,r)}(f-g)+M^{\sharp}_{B(x,r)}(g).
$$
From this, the definition of $\rho^\epsilon_{\sharp}(f)$ and Lemma \ref{l2.2}(iii),
we deduce that, for any $\epsilon\in(0,\fz)$ and $\delta\in(0,\fz)$,
there exists a positive constant $C_{(\delta)}$, depending on $\delta$, such that
\begin{equation}\label{2.18}
\rho^\epsilon_{\sharp}(f)\le C_{(\delta)}\rho^\epsilon_{\sharp}
(f-g)+(1+\delta)^\gamma\rho^\epsilon_{\sharp}(g),
\end{equation}
here and hereafter, $\gamma$ is as in Lemma \ref{l2.2}(ii).

Let $f\in\dot{W}^{1,\,\Phi}(\rn)$ and $g\in C^\fz_c(\rn)$.
Then, by \eqref{2.18}, Lemmas \ref{l2.1} and \ref{l2.3}, we conclude that
\begin{align}\label{2.19}
\limsup_{\epsilon\to0^+}\rho^\epsilon_{\sharp}(f)&\le C_{(\delta)}
\max\lf\{\||\nabla(f-g)|\|_{L^\Phi(\rn)},\,
\||\nabla(f-g)|\|_{L^\Phi(\rn)}^\gamma\r\}+(1+\delta)^\gamma\rho_{\Phi}(c_0|\nabla g|),
\end{align}
where $c_0$ is the same as in Theorem \ref{hr17thm}.
From Assumption \ref{a3}, it follows that $C^\fz_c(\rn)$ is
dense in $\dot{W}^{1,\,\Phi}(\rn)$,
which implies that there exists a sequence $\{g_i\}_{i\in\nn}\subset C^\fz_c(\rn)$
such that $\nabla g_i$ converges to $\nabla f$ in $L^\Phi(\rn)$ as $i\to\fz$.
Replacing $g$ by $g_i$ in \eqref{2.19} and letting $i\to\fz$, we obtain
\begin{equation}\label{2.20}
\limsup_{\epsilon\to0^+}\rho^\epsilon_{\sharp}(f)\le \lim_{i\to\fz}
(1+\delta)^\gamma\rho_{\Phi}(c_0|\nabla g_i|)=(1+\delta)^\gamma\rho_{\Phi}(c_0|\nabla f|).
\end{equation}
Letting $\delta\to0^+$ in \eqref{2.20}, we conclude that
\begin{equation}\label{2.21}
\limsup_{\epsilon\to0^+}\rho^\epsilon_{\sharp}(f)\le\rho_{\Phi}(c_0|\nabla f|).
\end{equation}
Moreover, similarly to \eqref{2.18}, we find that, for any $\epsilon\in(0,\fz)$
and $\delta\in(0,\fz)$, there exists a positive constant $C_{(\delta)}$, depending on $\delta$, such that
\begin{equation*}
\rho^\epsilon_{\sharp}(f)\ge\frac{1}{(1+\delta)^\gamma}\rho^\epsilon_{\sharp}(g_i)
-\frac{C_{(\delta)}}{(1+\delta)^\gamma}\rho^\epsilon_{\sharp}(f-g_i),
\end{equation*}
where, for any $i\in\nn$, $g_i$ is as in \eqref{2.20}. By this estimate and
similarly to \eqref{2.21}, we conclude that
$$\liminf_{\epsilon\to0^+}\rho^\epsilon_{\sharp}(f)\ge\rho_{\Phi}(c_0|\nabla f|),
$$
which, together with \eqref{2.21}, further implies that
$$\lim_{\epsilon\to0^+}\rho^\epsilon_{\sharp}(f)=\rho_{\Phi}(c_0|\nabla f|).
$$
Thus, \eqref{1.2} holds true.

\emph{Step 2)} In this step, we show that, if \eqref{1.1} holds true, then $|\nabla f|\in L^\Phi(\rn)$.

Let $G$ be as in \eqref{2.17}. By the triangle inequality on $\cc$
and a change of integration order, we find that, for any
$x\in\rn$, $r\in(0,1)$ and $\delta\in(0,\fz)$,
\begin{align}\label{2.22}
M^{\sharp}_{B(x,r)}(G_\delta\ast f)&=\frac{1}{|B(x,r)|}\int_{B(x,r)}
\lf|(G_\delta\ast f)(y)-(G_\delta\ast f)_{B(x,r)}\r|\,dy\\
&=\frac{1}{|B(x,r)|}\int_{B(x,r)}|G_\delta\ast(f-f_{B(\cdot,r)})(y)|\,dy
\le G_\delta\ast M^{\sharp}_{B(\cdot,r)}(f)(x).\nonumber
\end{align}
For any $\delta\in(0,\fz)$, $x\in\rn$ and $r\in(0,1]$, let
$$h_{r,\,\delta}(x):=\frac{1}{r}M^{\sharp}_{B(x,r)}(G_\delta\ast f)\ \
\text{and}\ \ g_r(x):=\frac{1}{r}M^{\sharp}_{B(x,r)}(f).$$
If $\|h_{r,\,\delta}\|_{L^\Phi(\rn)}\le1$, then, from Lemma \ref{l2.2}(i), it follows that
\begin{align}\label{2.23}
\int_{\rn}\Phi\lf(x,h_{r,\,\delta}(x)\r)\,dx&=\int_{\rn}\Phi\lf(x,\frac{h_{r,\,\delta}(x)}
{\|h_{r,\,\delta}\|_{L^\Phi(\rn)}}\|h_{r,\,\delta}\|_{L^\Phi(\rn)}\r)\,dx
\le\|h_{r,\,\delta}\|_{L^\Phi(\rn)}\le1.
\end{align}
If $\|h_{r,\,\delta}\|_{L^\Phi(\rn)}\in(1,\fz)$, by Lemma \ref{l2.2}(ii), \eqref{2.22}
and Assumption \ref{a4}, we find that
\begin{align}\label{2.24}
\int_{\rn}\Phi\lf(x,h_{r,\,\delta}(x)\r)\,dx&=\int_{\rn}\Phi\lf(x,\frac{h_{r,\,\delta}(x)}
{\|h_{r,\,\delta}\|_{L^\Phi(\rn)}}\|h_{r,\,\delta}\|_{L^\Phi(\rn)}\r)\,dx\\
&\le\|h_{r,\,\delta}\|_{L^\Phi(\rn)}^\gamma\le\|G_\delta\ast g_r\|_{L^\Phi(\rn)}^\gamma
\ls\|M(g_r)\|_{L^\Phi(\rn)}^\gamma\ls\|g_r\|_{L^\Phi(\rn)}^\gamma,\nonumber
\end{align}
which, further implies that, if $\|h_{r,\,\delta}\|_{L^\Phi(\rn)}\in(1,\fz)$,
then $\|g_r\|_{L^\Phi(\rn)}\gs1$.
From this and Lemma \ref{l2.2}(i), we deduce that
$$\|g_r\|_{L^\Phi(\rn)}\ls\int_{\rn}\Phi\lf(x,g_r(x)\r)\,dx.
$$
By this and \eqref{2.24}, we conclude that,
if $\|h_{r,\,\delta}\|_{L^\Phi(\rn)}\in(1,\fz)$, then
\begin{align*}
\int_{\rn}\Phi\lf(x,h_{r,\,\delta}(x)\r)\,dx\ls\lf\{\int_{\rn}
\Phi\lf(x,g_r(x)\r)\,dx\r\}^\gamma,
\end{align*}
which, combined with \eqref{2.23} and the fact that, for any $\epsilon\in(0,\fz)$,
$\int_0^1\psi_\epsilon(r)\,dr=1$, further implies that,
for any $\delta,\,\epsilon\in(0,\fz)$,
\begin{align}\label{2.25}
&\int_0^1\lf\{\int_{\rn}\Phi\lf(x,\frac{1}{r}M^{\sharp}_{B(x,r)}
(G_\delta\ast f)\r)\,dx\r\}^{1/\gamma}\psi_\epsilon(r)\,dr\\
&\hs\ls\int_0^1\lf[\int_{\rn}\Phi\lf(x,\frac{1}{r}M^{\sharp}_{B(x,r)}(f)\r)\,dx
+1\r]\psi_\epsilon(r)\,dr\ls\rho^\epsilon_{\sharp}(f)+1.\nonumber
\end{align}

Assume that $g\in C^2(\rn)$.
Repeating the proof of Lemma \ref{l2.1}, we know that, for any $R\in(0,\fz)$,
\begin{align}\label{2.26}
&\lim_{\epsilon\to0^+}\int_0^1\lf\{\int_{B(\vec{0}_n,R)}
\Phi\lf(x,\frac{1}{r}M^{\sharp}_{B(x,r)}
(g)\r)\,dx\r\}^{1/\gamma}\psi_\epsilon(r)\,dr\\
&\hs=\lf\{\int_{B(\vec{0}_n,R)}\Phi\lf(x,c_0|\nabla g(x)|\r)\,dx\r\}^{1/\gamma}.\nonumber
\end{align}
Assume that $\lim_{\epsilon\to0^+}\rho^\epsilon_{\sharp}(f)<\fz$.
Noticing that $G_\delta\ast f\in C^2(\rn)$ for any $\delta\in(0,\fz)$,
by \eqref{2.25} and \eqref{2.26},
we conclude that, for any $R,\,\delta\in(0,\fz)$,
\begin{align*}
&\lf\{\int_{B(\vec{0}_n,R)}\Phi(x,c_0|\nabla(G_\delta\ast f)(x)|)\,dx\r\}^{1/\gamma}\\
&\hs=\lim_{\epsilon\to0^+}\int_0^1\lf\{\int_{B(\vec{0}_n,R)}
\Phi\lf(x,\frac{1}{r}M^{\sharp}_{B(x,r)}
(G_\delta\ast f)\r)\,dx\r\}^{1/\gamma}\psi_\epsilon(r)\,dr\\
&\hs\le\lim_{\epsilon\to0^+}\int_0^1\lf\{\int_{\rn}\Phi\lf(x,\frac{1}{r}M^{\sharp}_{B(x,r)}
(G_\delta\ast f)\r)\,dx\r\}^{1/\gamma}\psi_\epsilon(r)\,dr
\ls\lim_{\epsilon\to0^+}\rho^\epsilon_{\sharp}(f)+1<\fz.
\end{align*}
Therefore, $\{\nabla(G_\delta\ast f)\}_{\delta>0}$ is a bounded sequence in $L^\Phi(\rn)$.
By this and Lemma \ref{l2.4}, combined with the well-known
Eberlein--$\check{\mathrm S}$mulian theorem, we conclude that there exists a subsequence of
$\{\nabla(G_\delta\ast f)\}_{\delta>0}$ weakly
converging in $L^\Phi(\rn)$ to a function $h\in L^\Phi(\rn)$,
which, together with the definition
of the derivative and the fact that $G_\delta\ast f$ converges to $f$ in $L^1_\loc(\rn)$,
implies that $h=\nabla f$. Thus, $|\nabla f|\in L^\Phi(\rn)$.

\emph{Step 3)} Finally, we show that \eqref{1.2} implies \eqref{1.3}.

Let $\delta\in(0,\fz)$ and
$$g:=\frac{f}{c_0(1+\delta)(\||\nabla f|\|_{L^\Phi(\rn)}+\delta)},$$
where $c_0$ is the same as in Theorem \ref{hr17thm}. From \eqref{1.2} and Lemma \ref{l2.2}(i), it follows that
\begin{align*}
\lim_{\epsilon\to0^+}\rho^\epsilon_{\sharp}(g)&=\rho_{\Phi}(c_0|\nabla g|)=
\rho_{\Phi}\lf(\frac{|\nabla f|}{(1+\delta)(\||\nabla f|\|_{L^\Phi(\rn)}+\delta)}\r)\\
&\le\frac{1}{1+\delta}\rho_{\Phi}\lf(\frac{|\nabla f|}
{\||\nabla f|\|_{L^\Phi(\rn)}+\delta}\r)\le
\frac{1}{1+\delta}.
\end{align*}
Thus, for any sufficiently small $\epsilon\in(0,\fz)$,
$\rho^\epsilon_{\sharp}(g)\le1$, which implies that
$\|g\|^{\epsilon}_{\sharp,\,\Phi}\le1$ and hence
$$\|f\|^{\epsilon}_{\sharp,\,\Phi}\le c_0(1+\delta)(\||\nabla f|\|_{L^\Phi(\rn)}+\delta).
$$
Letting $\delta\to0^+$, we obtain
$$\lim_{\epsilon\to0^+}\|f\|^{\epsilon}_{\sharp,\,\Phi}\le c_0\||\nabla f|\|_{L^\Phi(\rn)}.
$$
Similarly, we also have
$$\lim_{\epsilon\to0^+}\|f\|^{\epsilon}_{\sharp,\,\Phi}\ge c_0\||\nabla f|\|_{L^\Phi(\rn)}.
$$
Therefore, \eqref{1.3} holds true. This finishes the proof of Theorem \ref{t1.1}.
\end{proof}

Now we show Corollary \ref{c1.2}.

\begin{proof}[Proof of Corollary \ref{c1.2}]
To prove Corollary \ref{c1.2}, we only need to show that the Musielak--Orlicz function
$\Phi$ as in Corollary \ref{c1.2} satisfies Assumptions \ref{a1} through \ref{a4} and
its complementary function $\Phi^\ast$ satisfies Assumptions \ref{a1} and \ref{a2}.
We divide the proof into the following four steps according to the type of $\Phi$.

\emph{Step i)} In this step, we show Corollary \ref{c1.2}(i).
To this end, for any $x\in\rn$ and $t\in[0,\fz)$, let $\Phi(x,t):=t^{p(x)}$, where $p(\cdot)$
is as in Theorem \ref{hr17thm}.
In this case, by the assumption that $p$ satisfies the log-H\"older
continuity condition, we easily know that $p^+<\fz$ with $p^+$ as in
\eqref{p-1}, which, together with the assumption $p^-\in(1,\fz)$
with $p^-$ as in \eqref{p-1}, further implies that $1<p^-\le p^+<\fz$.
From this, we deduce that
$\Phi$ satisfies Assumptions \ref{a1} and \ref{a2}.
Moreover, it is known that Assumption \ref{a3}
holds true for such a function $\Phi$ (see, for example, \cite[Theorem 9.1.6]{dhhr11}).
Furthermore, from \cite[Theorem 4.3.8]{dhhr11} (see also \cite{d04,d05,dhhms09,h15}),
it follows that the Hardy--Littlewood maximal operator $M$ is bounded on $L^{p(\cdot)}(\rn)$,
which further implies that $\Phi$ satisfies Assumption \ref{a4}.
Thus, Assumptions \ref{a1} through \ref{a4} hold true for such a $\Phi$.

Moreover, it is easy to see that, for any $x\in\rn$ and
$t\in[0,\fz)$,
\begin{equation}\label{2.27}
\Phi^\ast(x,t)=\frac{1}{q(x)[p(x)]^{q(x)/p(x)}}t^{q(x)},
\end{equation}
where, for any $x\in\rn$,
$q(x)$ is given by the equality $\frac{1}{p(x)}+\frac{1}{q(x)}=1$.
Indeed, by the Young inequality, we know that, for any $k\in(1,\fz)$ and $a,\,b\in[0,\fz)$,
\begin{equation}\label{2.28}
ab\le\frac{a^k}{k}+\frac{b^{k'}}{k'},
\end{equation}
where $k'\in(1,\fz)$ is given by the equality $\frac{1}{k}+\frac{1}{k'}=1$, which,
combined with choosing $a:=s[p(x)]^{1/p(x)}$, $b:=t/[p(x)]^{1/p(x)}$ and $k:=p(x)$ in \eqref{2.28},
further implies that, for any $x\in\rn$ and $s,\,t\in[0,\fz)$,
$$st\le s^{p(x)}+\frac{1}{q(x)[p(x)]^{q(x)/p(x)}}t^{q(x)}.
$$
From this and the definition of $\Phi^\ast$, it follows that,
for any $x\in\rn$ and $t\in[0,\fz)$,
\begin{equation}\label{2.29}
\Phi^\ast(x,t)\le\frac{1}{q(x)[p(x)]^{q(x)/p(x)}}t^{q(x)}.
\end{equation}
Furthermore, by the condition $a^k=b^{k'}$ that is to guarantee
that the equality holds true in \eqref{2.28},
we conclude that, for any $x\in\rn$ and $t\in[0,\fz)$,
$$\Phi^\ast(x,t)\ge\frac{1}{q(x)[p(x)]^{q(x)/p(x)}}t^{q(x)},
$$
which, combined with \eqref{2.29}, further implies that \eqref{2.27} holds true
for any $x\in\rn$ and $t\in[0,\fz)$.
From the fact that $1<p^-\le p^+<\fz$, we deduce that $1<q^-\le q^+<\fz$,
which further implies that $\Phi^\ast$ satisfies Assumptions \ref{a1} and \ref{a2}.
Therefore, the conclusions of Theorem \ref{t1.1} and
Corollary \ref{c1.1} hold true for such a $\Phi$. This finishes the proof of Corollary \ref{c1.2}(i).

\emph{Step ii)} In this step, we show Corollary \ref{c1.2}(ii). To this end,
for any $x\in\rn$ and $t\in[0,\fz)$, let $\Phi(x,t):=\fai(t)$, where $\fai$ is an
Orlicz function satisfying the $\Delta_2$-condition and \eqref{1.4}.
Obviously, such a $\Phi$ satisfies Assumptions \ref{a1} and \ref{a2}. It is well known that
$C^\fz_c(\rn)$ is dense in the homogeneous Orlicz--Sobolev space
$\dot{W}^{1,\,\fai}(\rn)$ (see, for example, \cite[Theorem 8.31]{af03}),
which implies that the function
$\Phi$ satisfies Assumption \ref{a3}. From \cite[Theorem 2.1]{g88}
(see also \cite[Theorem 1.2.1]{kk91}),
we deduce that the Hardy--Littlewood maximal operator $M$ is bounded on $L^{\fai}(\rn)$,
which implies that such a $\Phi$ satisfies Assumption \ref{a4}.
Thus, Assumptions \ref{a1} through \ref{a4} hold true for such a $\Phi$.

Furthermore, by the fact that the complementary function $\Phi^\ast$
is independent of the spatial variable $x$, we conclude that
Assumption \ref{a1} holds true for $\Phi^\ast$.
From \eqref{1.4} and \cite[p.\,25, Theorem 4.2]{kr61}, it follows that $\Phi^\ast$
satisfies the $\Delta_2$-condition. Therefore,
Theorem \ref{t1.1} and Corollary \ref{c1.1} hold true for such a $\Phi$.
which completes the proof of Corollary \ref{c1.2}(ii).

\emph{Step iii)} In this step, we show Corollary \ref{c1.2}(iii).
To this end, for any $x\in\rn$ and $t\in[0,\fz)$, let $\Phi(x,t):=\omega(x)t^{p}$,
where $p\in(1,\fz)$ and $\omega\in A_p(\rn)$.
In this case, the Musielak--Orlicz space $L^{\Phi}(\rn)$ and
the homogeneous Musielak--Orlicz--Sobolev
space $\dot{W}^{1,\,\Phi}(\rn)$ are just the weighted Lebesgue space
$L^p_\omega(\rn)$ and the homogeneous weighted Sobolev space
$\dot{W}^{1,\,p}_\omega(\rn)$, respectively.
Recall that the \emph{weighted Lebesgue space} $L^p_\omega(\rn)$ is defined
to be the space of all Lebesgue measurable functions $f$ such that
$$\|f\|_{L^p_\omega(\rn)}:=\lf[\int_\rn|f(x)|\omega(x)\,dx\r]^{1/p}<\fz$$
and the \emph{homogeneous weighted Sobolev space} $\dot{W}^{1,\,p}_\omega(\rn)$
is defined to be the space of all $f\in L^1_\loc(\rn)$ such that $|\nabla f|\in L^p_\omega(\rn)$.
It is easy to see that $\Phi$ satisfies Assumptions \ref{a1} and \ref{a2}. Moreover,
from \cite[Theorem 2.1.4]{t00}, we deduce that $C^\fz_c(\rn)$ is dense in
the homogeneous weighted Sobolev space $\dot{W}^{1,\,p}_\omega(\rn)$,
which implies that Assumption \ref{a3} holds true for $\Phi$.
Furthermore, it is well known that the Hardy--Littlewood maximal
operator $M$ is bounded on the weighted
space $L^p_\omega(\rn)$ (see, for example, \cite[Theorem 7.1.9]{g14}),
which further implies that $\Phi$ satisfies Assumption \ref{a4}.
Thus, Assumptions \ref{a1} through \ref{a4} hold true for such a $\Phi$.

Moreover, similarly to \eqref{2.27}, we find that, for any $x\in\rn$ and $t\in[0,\fz)$,
$$\Phi^\ast(x,t)=\frac{1}{p'}p^{-\frac{p'}{p}}[\omega(x)]^{-\frac{p'}{p}}t^{p'},
$$
where $p'\in(1,\fz)$ is given by the equality $\frac{1}{p}+\frac{1}{p'}=1$.
It is easy to see that Assumption \ref{a2} holds true for $\Phi^\ast$.
Furthermore, from the fact that $\omega\in A_p(\rn)$ and the
properties of $A_p(\rn)$-weights, we deduce that $\omega^{-\frac{p'}{p}}\in A_{p'}(\rn)$
(see, for example, \cite[Proposition 7.1.5(4)]{g14}),
which implies that $\omega^{-\frac{p'}{p}}\in L^1_\loc(\rn)$ and hence $\Phi^\ast$ satisfies
Assumption \ref{a1}. Therefore, the conclusions of Theorem \ref{t1.1} and Corollary \ref{c1.1} hold
true for such a $\Phi$. This finishes the proof of Corollary \ref{c1.2}(iii).

\emph{Step iv)} In this step, we show Corollary \ref{c1.2}(iv). To this end,
for any $x\in\rn$ and $t\in[0,\fz)$, let $\Phi(x,t):=t^{p}+\omega(x)t^{q}$,
where $1<p<q<\fz$ and $\omega\in A_q(\rn)$. Then it is easy to see that such a $\Phi$ satisfies
Assumptions \ref{a1} and \ref{a2}. Moreover, by \cite[Theorem 1.1]{agsy17},
we find that Assumption \ref{a3} holds true for $\Phi$. From the facts that
the Hardy--Littlewood maximal operator $M$ is bounded on both $L^p(\rn)$ and $L^q_\omega(\rn)$
(see, for example, \cite[Theorem 7.1.9]{g14}),
it follows that $\Phi$ satisfies Assumption \ref{a4}.
Thus, Assumptions \ref{a1} through \ref{a4} hold true for such a $\Phi$.

Moreover, similarly to \eqref{2.27},
we know that, for any $x\in\rn$ and $t\in[0,\fz)$,
$$\Phi^\ast(x,t)\sim\frac{1}{p'}p^{-\frac{p'}{p}}t^{p'}+\frac{1}{q'}q^{-\frac{q'}{q}}
[\omega(x)]^{-\frac{q'}{q}}t^{q'},
$$
where $p',\,q'\in(1,\fz)$  are given, respectively, by $\frac{1}{p}+\frac{1}{p'}=1$
and $\frac{1}{q}+\frac{1}{q'}=1$, and the implicit positive equivalence constants are independent of $x$
and $t$. By this and similarly to the corresponding proof in
Step iii), we conclude that both Assumptions \ref{a1} and \ref{a2} hold true for $\Phi^\ast$.
Therefore, Theorem \ref{t1.1} and Corollary \ref{c1.1} hold true for such a $\Phi$,
which completes the proof of Corollary \ref{c1.2}(iv) and hence of Corollary \ref{c1.2}.
\end{proof}

\bigskip

\noindent Sibei Yang

\medskip

\noindent School of Mathematics and Statistics, Gansu Key Laboratory of Applied Mathematics and
Complex Systems, Lanzhou University, Lanzhou 730000, People's Republic of China

\smallskip

\noindent{\it E-mail:} \texttt{yangsb@lzu.edu.cn}

\bigskip

\noindent Dachun Yang (Corresponding author) and Wen Yuan

\medskip

\noindent Laboratory of Mathematics and Complex Systems (Ministry of Education of China),
School of Mathematical Sciences, Beijing Normal University, Beijing 100875, People's Republic of China

\smallskip

\noindent{\it E-mails:} \texttt{dcyang@bnu.edu.cn} (D. Yang)

\noindent\phantom{{\it E-mails:} }\texttt{wenyuan@bnu.edu.cn} (W. Yuan)

\end{document}